\pdfoutput=1
\documentclass[]{article}  
\usepackage{url,float}
\usepackage{graphicx}
\usepackage{amsmath}
\usepackage{amsfonts}
\usepackage{amssymb}
\usepackage{latexsym}
\usepackage{caption}
\usepackage{xcolor}
\usepackage[colorlinks]{hyperref}
\hypersetup{
    colorlinks=true,
    linkcolor=blue,
    filecolor=magenta,      
    urlcolor=cyan,
    citecolor=purple
    }

\newcommand{\hide}[1]{}

\newcommand{\ABox}{
\raisebox{3pt}{\framebox[6pt]{\rule{6pt}{0pt}}}
}
\newenvironment{proof}{{\bf Proof:}}{\hfill\ABox}

\newtheorem{thm}{{\bf Theorem}}
\newtheorem{cor}{Corollary}
\newtheorem{lem}{Lemma}

\newtheorem{claim}{Claim}

\newcommand{\lemlab}[1]{\label{lemma:#1}}
\newcommand{\thmlab}[1]{\label{therom:#1}}

\newcommand{\corlab}[1]{\label{cor:#1}}
\newcommand{\tablab}[1]{\label{tabl:#1}} %something wrong, not sure what
\newcommand{\figlab}[1]{\label{fig:#1}}
\newcommand{\seclab}[1]{\label{sec:#1}}

\newcommand{\lemref}[1]{\ref{lemma:#1}}
\newcommand{\thmref}[1]{\ref{therom:#1}}
\newcommand{\corref}[1]{\ref{cor:#1}}
\newcommand{\figref}[1]{\ref{fig:#1}}
\newcommand{\tabref}[1]{\ref{tabl:#1}}

\def\a{{\alpha}}

\def\q{{\theta}}
\def\int{{\operatorname{int}}}

\newcommand{\squeezelist}{\setlength{\itemsep}{0pt}}
\def\defn#1{\textit{\textbf{\boldmath #1}}}

\usepackage{hyperref}
\hypersetup{
    colorlinks=true,
    linkcolor=blue,
    filecolor=magenta,      
    urlcolor=cyan,
    }

\usepackage[shortlabels]{enumitem}

{\makeatletter
\gdef\@fnsymbol#1{\ensuremath{\ifcase#1\or *\or \dagger\or \ddagger\or
   \mathsection\or \mathparagraph\or \|\or **\or \dagger\dagger\or
   \ddagger\ddagger\or \mathsection\mathsection\or
   \mathparagraph\mathparagraph\or \|\|\or ***\or \dagger\dagger\dagger\or
   \ddagger\ddagger\ddagger\or \mathsection\mathsection\mathsection\or
   \mathparagraph\mathparagraph\mathparagraph\or \|\|\|\else\@ctrerr\fi}}}

\title{%
Quasigeodesics on the Cube
} %title

\author{%
MIT CompGeom Group\thanks{%
Artificial first author to highlight that the other authors
(in alphabetical order) worked as an equal group.
}
\and
Hugo A. Akitaya\thanks{%
U. Mass. Lowell, \texttt{hugo\_akitaya@uml.edu}}
\and
Erik D. Demaine\thanks{%
MIT, \texttt{edemaine@mit.edu}}
\and
Adam Hesterberg\thanks{%
Harvard U., \texttt{ahesterberg@seas.harvard.edu}}
\and
Thomas C. Hull\thanks{%
\texttt{thomas.hull@fandm.edu}}
\and
Anna Lubiw\thanks{%
U. Waterloo, \texttt{alubiw@uwaterloo.ca}}
\and
Jayson Lynch\thanks{%
MIT, \texttt{jaysonl@mit.edu}}
\and
Klara Mundilova\thanks{%
\texttt{klara.mundilova@gmail.com}}
\and
Chie Nara\thanks{
\texttt{cnara@jeans.ocn.ne.jp}}
\and
Joseph O'Rourke\thanks{%
Smith College, \texttt{jorourke@smith.edu}}
\and
Frederick Stock\thanks{%
U. Mass. Lowell, \texttt{fbs9594@rit.edu}} 
\and
Josef Tkadlec\thanks{%
\texttt{josef.tkadlec@iuuk.mff.cuni.cz}}
\and
Ryuhei Uehara\thanks{%
\texttt{uehara@jaist.ac.jp}}
}%author

\date{\today}

\begin{document}
\maketitle

%\setcounter{tocdepth}{3}
%\tableofcontents

\begin{abstract}
A quasigeodesic 
is a curve on the surface of a convex polyhedron that
has $\le \pi$ surface to each side at every point.
In contrast, a geodesic has exactly $\pi$ to each side and so can never pass through
a vertex, whereas quasigeodesics can.
Although it is known that every convex polyhedron
has at least three simple closed quasigeodesics, little else is known.
Only tetrahedra have been thoroughly studied.

\medskip
In this paper we explore the quasigeodesics on a cube, which have not
been previously enumerated.
We prove that the cube has exactly $15$ simple closed quasigeodesics
(beyond the three known simple closed geodesics).
For the lower bound we detail $15$ simple closed quasigeodesics.  Our main contribution is  establishing a matching upper bound.
For general convex polyhedra, there is no known upper bound.  
\end{abstract}

\section{Introduction}
\seclab{Introduction}
\subsection{Quasigeodesics}
\seclab{Quasigeodesics}
A \defn{quasigeodesic}
is a curve on the surface of a convex polyhedron that
has $\le \pi$ surface to each side at every point.
In contrast, a \defn{geodesic} has exactly $\pi$ to each side.
Because a vertex is a point with less than $2\pi$ surface, 
a geodesic can never pass through a vertex.
%\Anna{This is more directly from the definition and avoids issues of shortest vs locally shortest.}
%\JOR{Nice!}
Quasigeodesics can pass through vertices.

Ever since since Poincar\'e's investigations more
than a century ago,
closed geodesics have played an important role in the
topology of Riemannian manifolds~\cite[p.~433]{b-pvrg-03}.
It is a famous $1929$ theorem of Lyusternik-Schnirelmann that every smooth genus-$0$ surface
has at least three simple (non-self-intersecting)
closed geodesics~\cite{ls-optgf-29}.
Pogorelov proved in $1949$ a natural analog:
Every convex surface has at
%\JOR{His therom does not need polyedron.}
least three simple closed quasigeodesics~\cite{p-qglcs-49}.
Pogorelov's existence proof does not suggest a way to identify the three 
quasigeodesics, and it is only recently that finite algorithms have been 
proposed~\cite{demaine2020quasi}~\cite{ChartierdeMesmay}.

Aside from these algorithms,
simple closed quasigeodesics have only been systematically studied on tetrahedra.
Two results in~\cite{QonT} are: 
(1)~On any non-isosceles tetrahedron, there is at least one $1$-vertex, 
one $2$-vertex,
and one $3$-vertex 
simple closed quasigeodesic.
%(Isosceles do not have $1$-vertex simple closed quasigeodesics, but they
%have simple closed geodesics.)
%(1)~On any tetrahedron, there is
%at least one $2$-vertex, at least one $3$-vertex 
%simple closed quasigeodesic,
%and at least one simple closed quasigeodesic through at most $1$ vertex.
(2)~There is an open set in the space of all tetrahedra, each element of which has at least $34$ simple closed quasigeodesics.
In contrast to (1),
it is known
from~\cite{davis2017geodesics} that the cube does not have a $1$-vertex simple
closed quasigeodesic.

Simple closed quasigeodesics play central roles in~\cite{QuasiTwist} and~\cite{Reshaping}, and are of interest in their own right.
But beyond their existence, much remains unknown. 
There is no known upper bound on the number of simple closed quasigeodesics
on a given polyhedron, and
there is an $n$-vertex polyhedron with $2^{\Omega(n)}$ distinct simple closed quasigeodesics~\cite[Sec.~24.4]{do-gfalop-07}. %p347
In contrast, it is known that isosceles 
% Do we want to mention any other common names? tetramonohedron and isotetrahedron got brought up in the meeting, Wikipedia also lists bisphenoid, equifacial tetrahedron, almost regular tetrahedron https://en.wikipedia.org/wiki/Disphenoid
tetrahedra\footnote{Also called disphenoids,
tetramonohedra,
isotetrahedra, and several other names.
All faces are congruent acute triangles.} 
% Do we want to mention any other common names? tetramonohedron and isotetrahedron got brought up in the meeting, Wikipedia also lists bisphenoid, equifacial tetrahedron, almost regular tetrahedron https://en.wikipedia.org/wiki/Disphenoid
have arbitrarily long 
``spiraling'' simple closed geodesics~\cite{protasov2007closed}~\cite{akopyan2018long}.
%And
%``there is no known upper bound on the combinatorial complexity of a simple closed quasigeodesic''~\cite{ChartierdeMesmay}.

In this paper we make a complete inventory of simple closed quasigeodesics on a cube.
It was known that there are precisely three simple closed geodesics on the cube.
We identify a further $15$ simple closed quasigeodesics (up to symmetries), and prove that this list is complete. 
We consider this proof to be our most significant contribution.

\subsection{Three simple closed geodesics}
\seclab{ThreeQgeos}
To describe geodesics and quasigeodesics explicitly, we adopt the notation
for faces and vertices displayed in
Fig.~\figref{CubeLabeling}.
Note that we label vertices in figures by their index $i$,
but refer to them in the text as $v_i$.
%%%%%%%%%%%%%%%%%%%%%%%%%%%%%%%%%Figure Begin
\begin{figure}[htbp]
\centering
\includegraphics[width=0.45\textwidth]{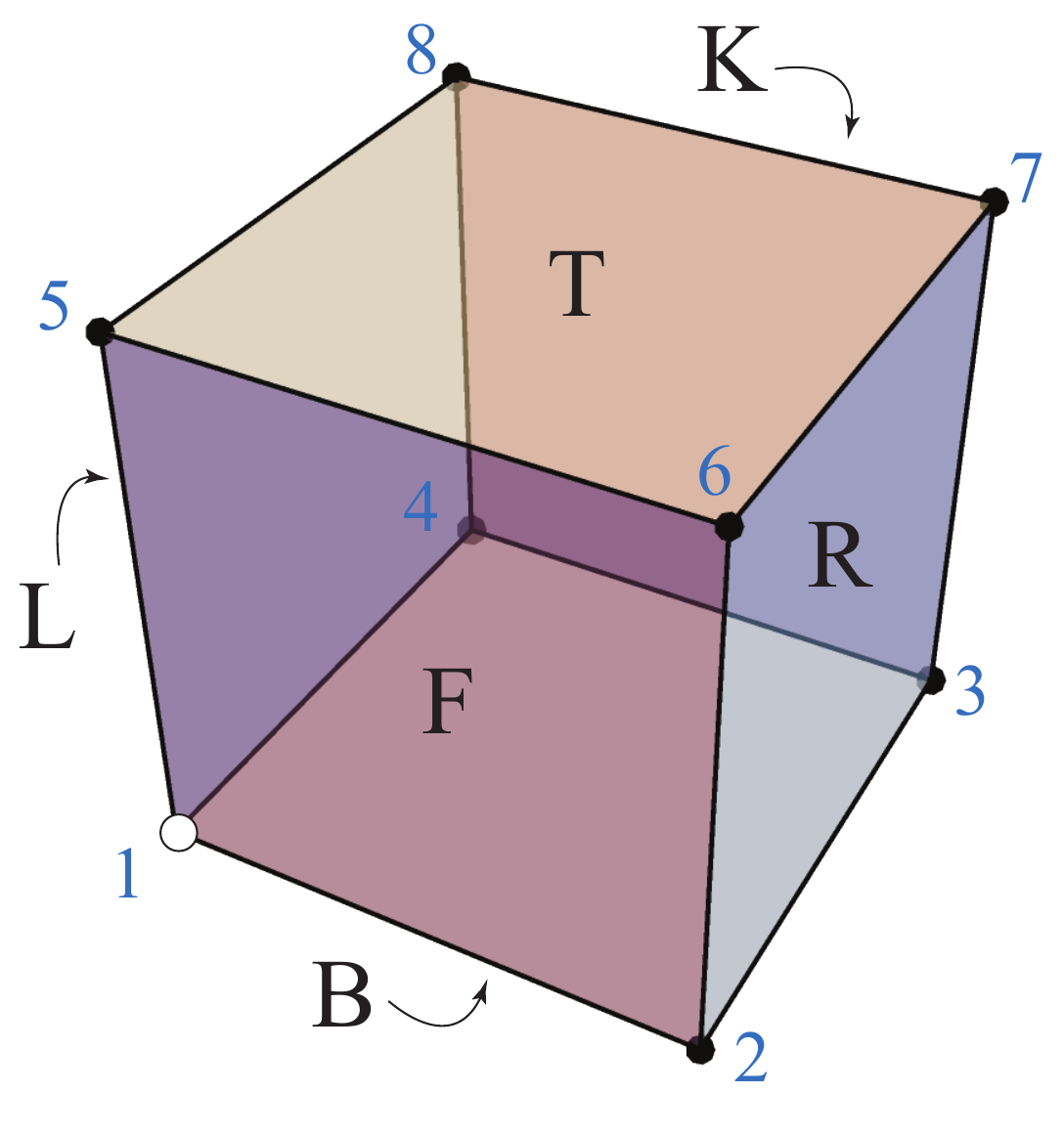}
\caption{F,R,T,K,L,B $=$ Front, Right, Top, bacK, Left, Bottom.
B vertices indexed $1,2,3,4$;
T vertices indexed $5,6,7,8$.
$v_1$ is marked white.
% \JOR{Note: I'm using textrm for faces, F not $F$.}
}
\figlab{CubeLabeling}
\end{figure}
%%%%%%%%%%%%%%%%%%%%%%%%%%%%%%%%%Figure End

It has long been known that there 
are precisely three simple closed geodesics on the cube~\cite{fuchs2007closed},
displayed in Fig.~\figref{Geos3_2D3D}.%
\footnote{Note these three are not the three
from Pogorelov's theorem.}
%%%%%%%%%%%%%%%%%%%%%%%%%%%%%%%%%Figure Begin
\begin{figure}[htbp]
\centering
\includegraphics[width=0.75\textwidth]{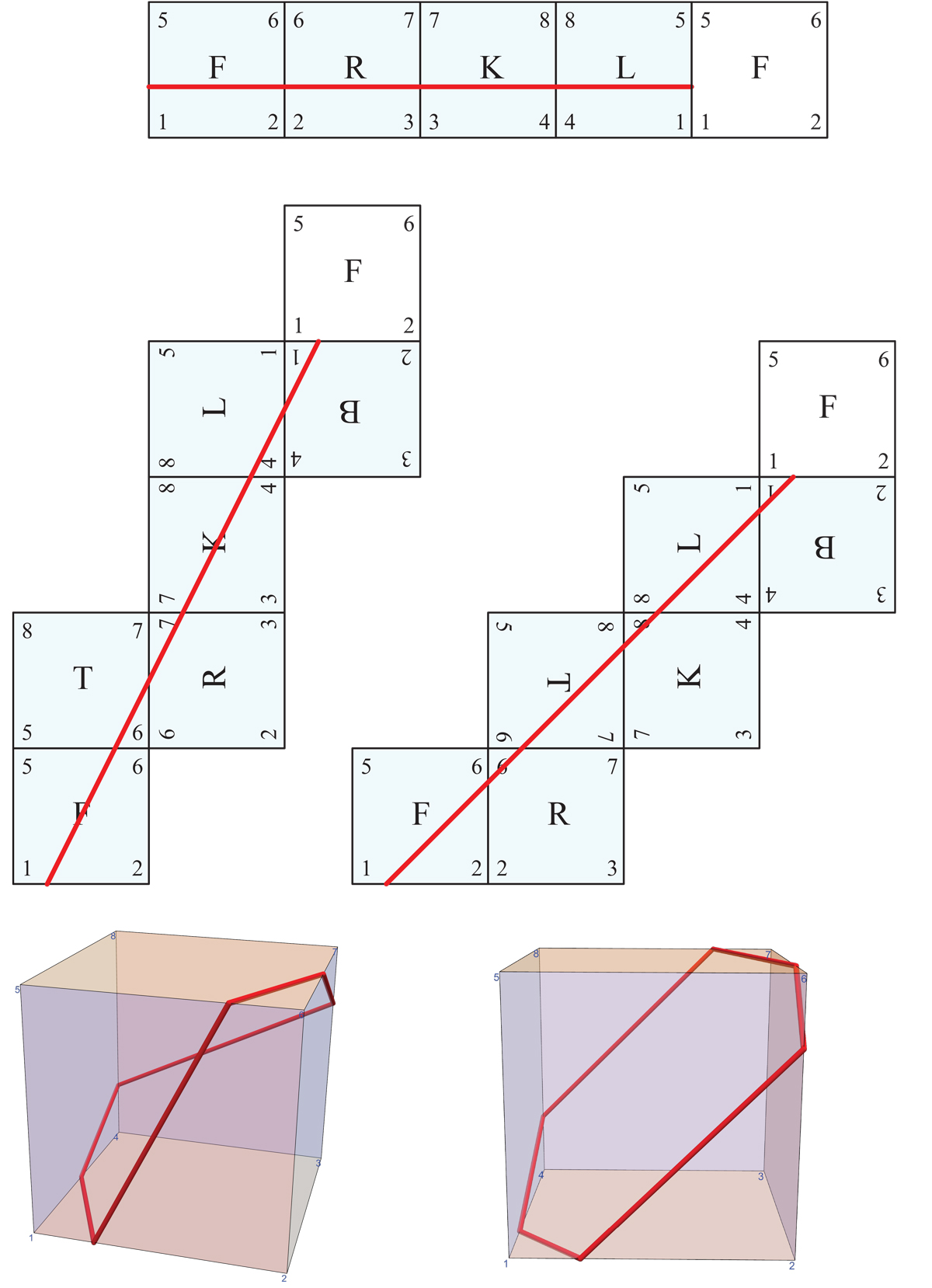}
\caption{The three simple closed geodesics on a cube.
The first is an equatorial band.
The other two are as depicted.
% \Anna{Note: this 3 is not the 3 in the proof -- those would count the three orientations of the top example because they do not factor out symmetries of the cube.}
% \jaysonl{we only depict two of the three. I know the third is obvious but we need to at least state it in the caption if not include a figure for completeness.}
% }
}
\figlab{Geos3_2D3D}
\end{figure}
%%%%%%%%%%%%%%%%%%%%%%%%%%%%%%%%%Figure End
Note that each of the three geodesics can 
slide
within a range, maintaining
parallelism. This is because each geodesic lies on a cylinder,
with $2 \pi$ curvature (four vertices, each with $\pi/2$ curvature) to each side. 
%\pepa{maybe ``angle $\pi$ to each side''?}
%\JOR{I think it is more than that there is pi angle to each point.}

%\clearpage
\section{Outline of Argument}
\seclab{Outline}
We mentioned that simple closed geodesics can spiral around isosceles tetrahedra.
A simple closed quasigeodesic also may spiral around other convex polyhedra, as shown in Fig.~\figref{LongBoxTurns} below.
% \JOR{In fact, even simple closed geodesics may spiral around an isoceles tetrahedron \Anna{citation} and around a doubly-covered square \Anna{add figure?}.
% }
A central aspect of our proof is to show that quasigeodesics  cannot spiral on a cube.

% \JOR{Revised a bit, because isos tetra already mentioned, and I think the doubly covered square is best saved to OpenProbs.}

Define a \defn{geodesic segment} as a non-self-intersecting
vertex-to-vertex geodesic.\footnote{
In some literature, a geodesic segment is a shortest path between its endpoints.
In this paper, our geodesic segments may or may not be shortest.}
A simple closed quasigeodesic is composed of a sequence of %geodesics 
geodesic segments,
satisfying the $\le \pi$ condition to both sides at each vertex.

An instructive example was identified in~\cite{demaine2020quasi}:
a long box with a spiraling simple closed quasigeodesic.
See Fig.~\figref{LongBoxTurns}.
Each of the four marked vertices has $\pi$ angle to one side and $\pi/2$ to the other side.
Since there is freedom to %alter these angles,
partition the $3\pi/2$ surface angle differently
(while maintaining $\le \pi$ to each side),
the number of spiraling simple closed quasigeodesics of a long box grows with the length of the long side of the box. 
% \JOR{I don't understand ``freedom to alter these angles.'' The diagonal of the square end must have 45deg on the square, which forces 45deg on the long side.}
% \Anna{Right now, the angle between green and blue at $v_1$ is $180$ on the left and $90$ on the right.  You can trade these off. Decrease the left angle and increase the right one.  So if the box were (roughly) twice as long, I think the choice of $180$ on the right makes you spiral twice as often as in the figure.  But you could choose a smaller right angle and spiral half as often.  I could try making this exact if you need me to.}
%We will rule out spiraling on the cube. 
A crucial property of spiraling is that some geodesic segment re-enters its initial face.  For example, the blue geodesic segment from $v_1$ to $v_2$ in the figure
starts on the long front-side face and later re-enters that face.
% \jor{We will see that this cannot happen on a cube:
% informally, a geodesic segment can cross a face at most once.}
%\Anna{Perhaps: 
We will prove that this cannot happen on a cube: a geodesic segment cannot return to its initial face, and in fact, cannot cross any face more than once.
% \Anna{NEW.  Actually, both follow right away from Lemma 1.  The stronger thing that follows only from our algorithm is that a quasigeo does not cross any face more than once.}
%\Anna{and then later on, make both these claims explicit -- the first claim following from analysis of slopes, and the second from examining the 15 solutions (so the second claim relies on the exhaustive search algorithm but the first claim does not).}
% \JOR{Excellent idea; implemented.}

%see Fig.~\figref{LongBoxTurns}.
%%%%%%%%%%%%%%%%%%%%%%%%%%%%%%%%%Figure Begin
\begin{figure}[htbp]
\centering
\includegraphics[width=0.75\textwidth] {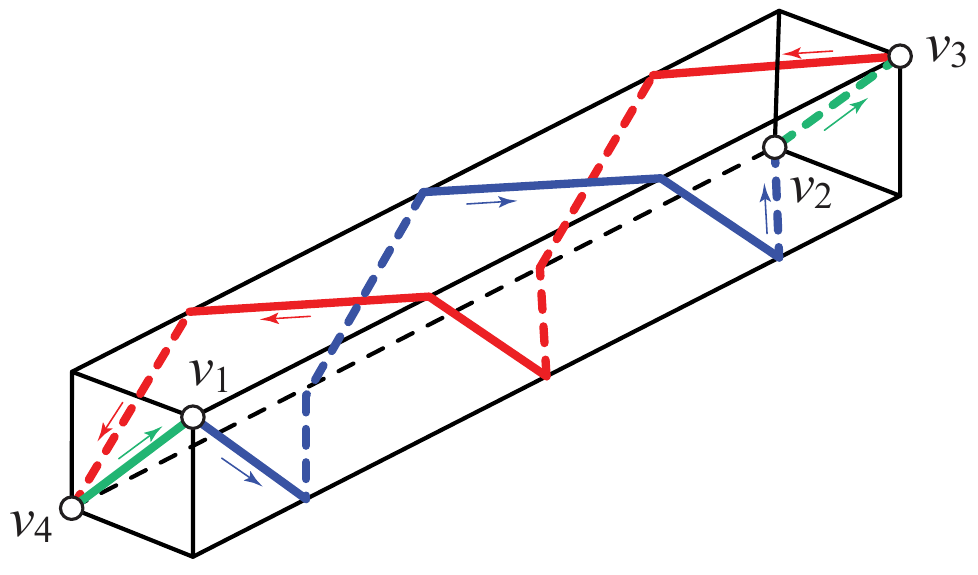}
\caption{$(v_1,v_2,v_3,v_3,v_4)$ is a simple closed quasigeodesic.
Based on Fig.~2 in~\cite{demaine2020quasi}.}
\figlab{LongBoxTurns}
\end{figure}
%%%%%%%%%%%%%%%%%%%%%%%%%%%%%%%%%Figure End

% \Anna{I suggest removing Lemma 1.  Maybe add, as an outline of our lower bound technique: Our approach is to analyze a geodesic segment based on the angle it makes in its first face.  We first rule out some angles because the geodesic segment revisits the first face and intersects itself there.  We rule out further angles by finding intersections between two consecutive geodesic segments.  This reduces the possible angles to a finite set, which allows a combinatorial enumeration of all simple closed quasigeodesics.}

%
% \JOR{To be removed below...}
% \begin{lem}
% \lemlab{5Slopes}
% Let $F$ be a face of the cube, and $g$ a geodesic segment
% that is part of a simple closed quasigeodesic.
% Then the number of segments in $g \cap \int(F)$,
% where $\int(F)$ is the interior of face $F$, is at most one.
% \end{lem}
% \noindent
% This lemma reduces the problem of finding all simple closed quasigeodesics
% to a finite search:
% There cannot be more than $8$ geodesic segments comprising a
% quasigeodesic, and each segment can cross at most $6$ faces.
% \Anna{I moved the following to the start of the section:}
% Without this lemma, spiraling could conceivably be unbounded,
% as it is for simple closed geodesics on isosceles tetrahedra.

\section{Fifteen Simple Closed Quasigeodesics}
\seclab{15Quasigeos}
Here is our main result:
\begin{thm}
\thmlab{main15}
There are exactly $15$ simple closed quasigeodesics on the cube
(beyond the three simple closed geodesics noted above). %, and precisely the quasigeodesics 
These are displayed in Fig.~\figref{InventoryPepa15}
and described in Table~1. %\tabref{Description15}.
%\JOR{Can't get ref to work. Hardwiring.}
\end{thm}
\noindent
As our sole focus in the remainder is on
``simple closed quasigeodesics,'' we often simplify that term 
to \defn{quasigeos}.

% minipage to force fig & table together.
\noindent\begin{minipage}{\linewidth}
% \captionof{figure}{Example figure caption (non-floating)}
% \label{fig:examplenf}
%Fig.~\figref{InventoryPepa15}.
%%%%%%%%%%%%%%%%%%%%%%%%%%%%%%%%%Figure Begin
%\begin{figure}[htbp]
\centering
\includegraphics[width=1.0\textwidth]{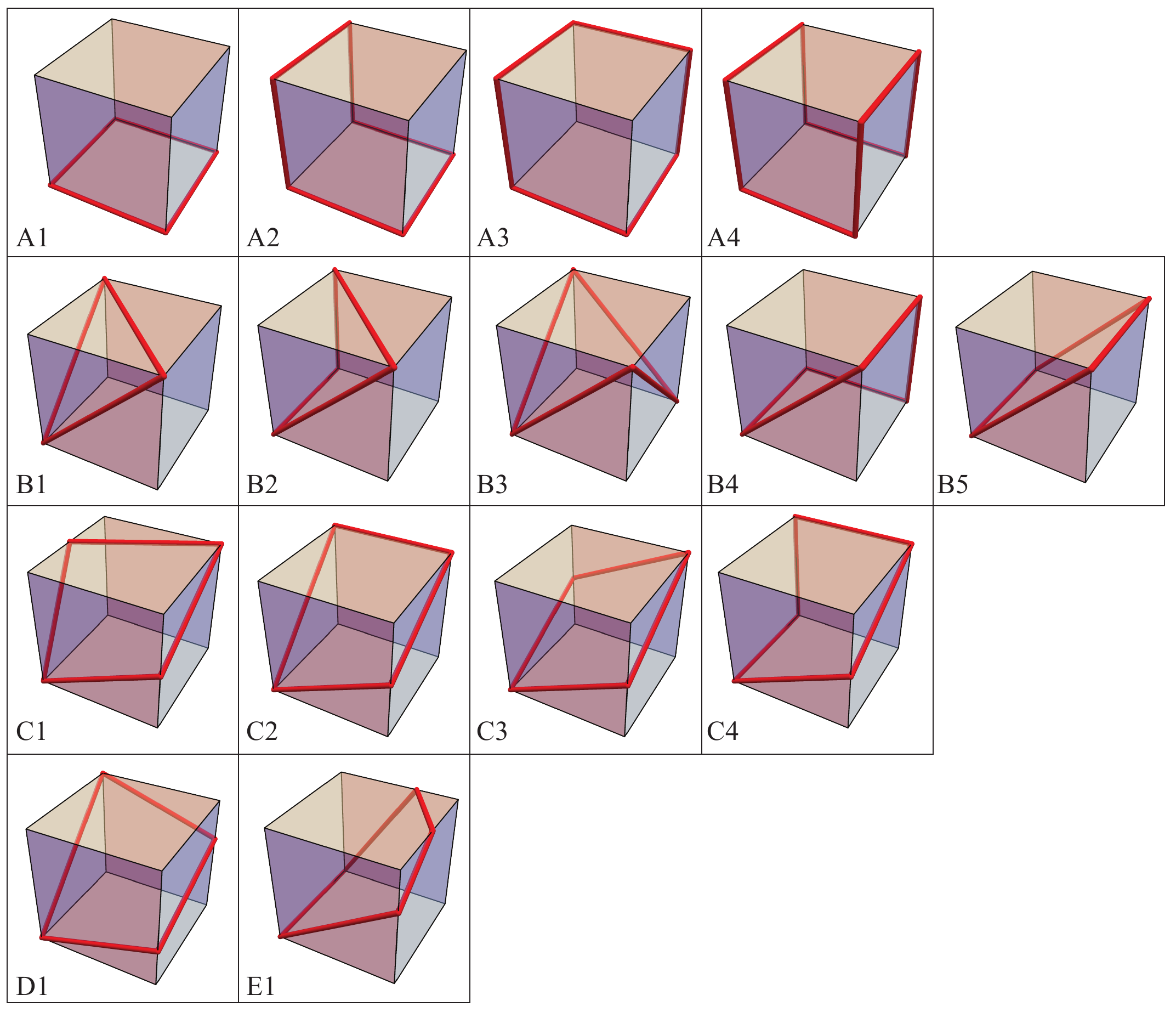}
\captionof{figure}{The $15$ simple closed quasigeodesics.}
\figlab{InventoryPepa15}
%\end{figure}
%%%%%%%%%%%%%%%%%%%%%%%%%%%%%%%%%Figure End
%
%\begin{table}
\begin{center}
\begin{tabular}{ | c l  |}
\hline
$A_i$: & Four quasigeos using only ${0/1}$ segments (each of length $1$).
\\
 $B_i$: & Five quasigeos using at least one ${1/1}$ segment (length $\sqrt{2}$), and none longer.
\\
 $C_i$: & Four quasigeos using at least one ${1/2}$ segment (length $\sqrt{5}$), and none longer.
\\
$D_1$: & One quasigeo using a single ${1/3}$ (length $\sqrt{10}$), and none longer.
\\
$E_1$: & One quasigeo using a single ${2/3}$ (length $\sqrt{13}$).
\\
\hline
\end{tabular}
\tablab{Description15} %Cannot get this to work....
\captionof{table}{Description of the five categories of quasigeos.}
\end{center}
%\end{table}
\end{minipage}

\bigskip

The quasigeos are listed in %the in
order of the length of the geodesic segments 
comprising them, as described in Table~\tabref{Description15}.
We identify a geodesic segment by its slope
$y/x$, i.e., vertically up $y$ units and rightward horizontally $x$ units within the 
natural coordinate system of its starting face.

\section{Five Slopes}

%\subsection{Overview}
Our approach is to analyze a geodesic segment  based on the
angle $\a$ 
%\JOR{$\a$. $\q$ is cone angle.}
%\Anna{oops, I get it now}
it makes in its starting face.
Consider a geodesic segment that does not follow an edge of the cube.  Then it enters the interior of a face and makes an angle in the range $(0, \pi/4]$ with one edge of the face. We express this as a slope in the range $(0,1]$.
We first rule out some slopes in this range because the geodesic
segment revisits the first face and intersects itself there. We rule out further
slopes by finding intersections between two %consecutive 
geodesic segments. This
reduces the possible slopes to a finite set, which allows a combinatorial enumeration of all simple closed quasigeodesics.

%see Fig.~\figref{FiveSlopes}.
%%%%%%%%%%%%%%%%%%%%%%%%%%%%%%%%%Figure Begin
\begin{figure}[htbp]
\centering
\includegraphics[width=0.5\textwidth]{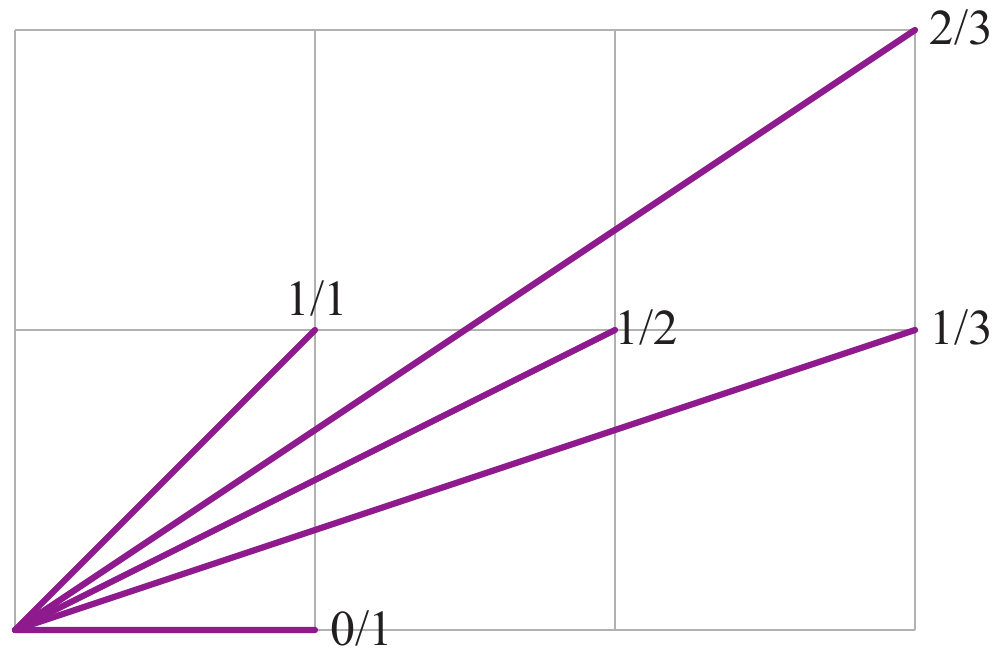}
\caption{The five possible distinct slopes.}
\figlab{FiveSlopes}
\end{figure}
%%%%%%%%%%%%%%%%%%%%%%%%%%%%%%%%%Figure End

\begin{lem}
\lemlab{5Slopes}
A geodesic segment that is a component of a simple
closed quasigeodesic on the cube can only have one of
the five slopes shown in Fig.~\figref{FiveSlopes}:
$0/1,1/3,1/2,2/3,1/1$. 

% \Anna{Putting the values in the lemma statement makes it possible for someone to quote our lemma.}
%\JOR{I'd prefer we use $0/1$ rather than $0$, one edge.
%Seems more consistent? But not a big issue.}
%\Anna{done}
\end{lem}

\begin{cor}
\corlab{NotTwice}
A geodesic segment that is a component of a simple
closed quasigeodesic on the cube does not cross any face more than once.    
\end{cor}

We prove the lemma by partitioning the rest of the slope range $(0,1]$ into the following seven ranges:
\begin{itemize}
\item Case 1. $(0/1, 1/4]$
\item Case 2. $(1/4,1/3)$
\item Case 3. $(1/3,2/5)$
\item Case 4. $[2/5,1/2)$
\item Case 5. $(1/2,2/3)$
\item Case 6. $(2/3,3/4)$
\item Case 7. $[3/4,1/1)$
\end{itemize}

Fig.~\figref{CasesFill45deg} shows the seven cases, and 
Fig.~\figref{SevenCases} shows how each range progresses on the unfolded surface of the cube.
Each case has a (pink) \defn{F-cone} with
angle $\q$ at $v_1$.
From
Fig.~\figref{SevenCases} we immediately obtain:

%\Anna{The claims are new (but just re-packaging our proofs).}

\begin{claim} 
\label{claim:easy-cases}
No geodesic segment is possible in 
Cases 2, 3, and 6 because the segment revisits the starting face and intersects itself there. 
(We note that the crossing is at right angles, 
a known constraint~\cite{fuchs2007closed}.)
\end{claim}

The remaining four cases are possible for a single geodesic segment, but not for a geodesic segment that is part of a quasigeo.

\begin{claim} 
\label{claim:backward-cases}
Consider a geodesic segment $g$ that is a component of a simple
closed quasigeodesic on the cube, and that falls into Case 1, 4, 5, or 7.  Then 
$g$ intersects another segment of the quasigeodesic.
%the previous geodesic segment $g'$ intersects $g$. 
\end{claim}

\begin{proof}
We find an intersection point by following the quasigeo backwards from $v_1$, the starting vertex of $g$.
Let $g'$ be the geodesic segment before $g$. We trace $g'$ backwards from its terminus at $v_1$.
%\Anna{I've moved the 4 cases right here and deleting the text of Set-UP and Backward Analysis:Overview since they are already summarized by the above.}    

% %\clearpage
% \section{No Crossing a Face Twice}
% \seclab{NotTwice}
% %
% \anna{In this section we prove}
% %We now head toward proving 
% Lemma~\lemref{NotTwice}:
% no geodesic segment $g$ that is part of a quasigeo can intersect a face interior twice.

\hide{
\subsection{Set-Up}
Since we factor out symmetries, it
is no loss of generality to assume $g$ starts at $v_1$ and crosses
face F at a counterclockwise angle $\a$ with respect to $v_1 v_2$
in $(0,\pi/4]$. 
See again Fig.~\figref{CubeLabeling}.

% \jaysonl{we need a figure defining this notation, in particular $v_1$ and $F$. We can probably backrefrence Figure 1, but need to note that vertex $1$ is called $v_1$ (or change notation).}
% We identify a geodesic segment by its slope
% $y/x$, i.e., vertically up $y$ units and rightward horizontally $x$ units within the 
% natural coordinate system of its starting face F.
%So the range of slopes for $g$ is $0/1$ to $1/1$.

%see Fig.~\figref{CasesFill45deg}.
} % end hide
%%%%%%%%%%%%%%%%%%%%%%%%%%%%%%%%%Figure Begin
\begin{figure}[htbp]
\centering
\includegraphics[width=1.0\textwidth]{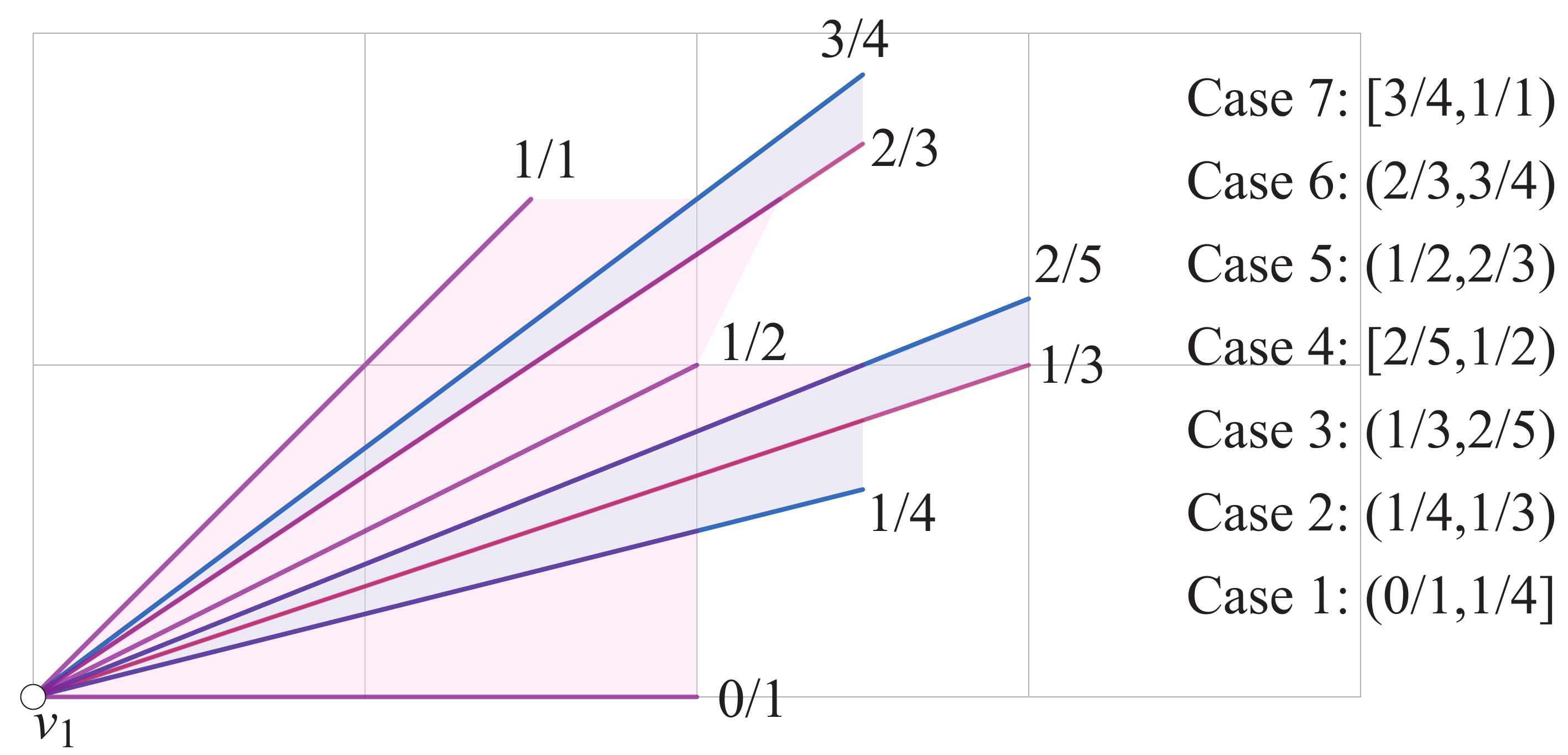}
\caption{The seven slope ranges.  Cases 2, 3, and 6 (in blue) are ruled out in Claim~\ref{claim:easy-cases}, and Cases 1, 4, 5, and 7 (in pink) are ruled out in Claim~\ref{claim:backward-cases}, leaving only the five slopes (in purple) allowed in Lemma~\lemref{5Slopes}.}
\figlab{CasesFill45deg}
\end{figure}
%%%%%%%%%%%%%%%%%%%%%%%%%%%%%%%%%Figure End

%see Fig.~\figref{SevenCases}.
%%%%%%%%%%%%%%%%%%%%%%%%%%%%%%%%%Figure Begin
\begin{figure}[htbp]
\centering
\includegraphics[width=1\textwidth]{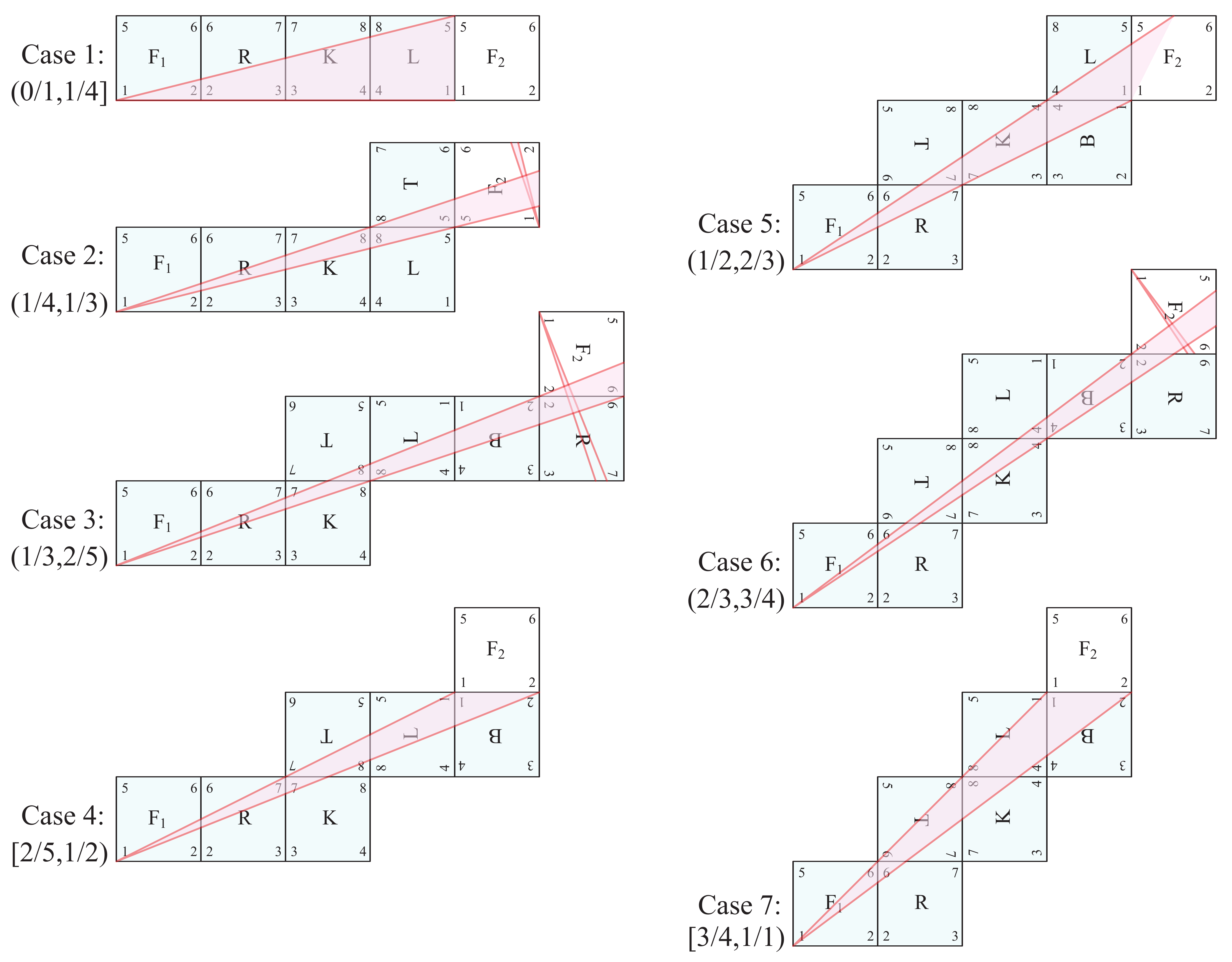}
\caption{The seven slope cases, showing the range of slopes (in pink) progressing across the faces of the cube.  The geodesic segment starts in face $F_1$ and revisits the starting face, marked $F_2$ (in white).
}
\figlab{SevenCases}
\end{figure}
%%%%%%%%%%%%%%%%%%%%%%%%%%%%%%%%%Figure End
%

%see Fig.~\figref{Cone3D14}.
%%%%%%%%%%%%%%%%%%%%%%%%%%%%%%%%%Figure Begin
\begin{figure}[htbp]
\centering
\includegraphics[width=0.5\textwidth]{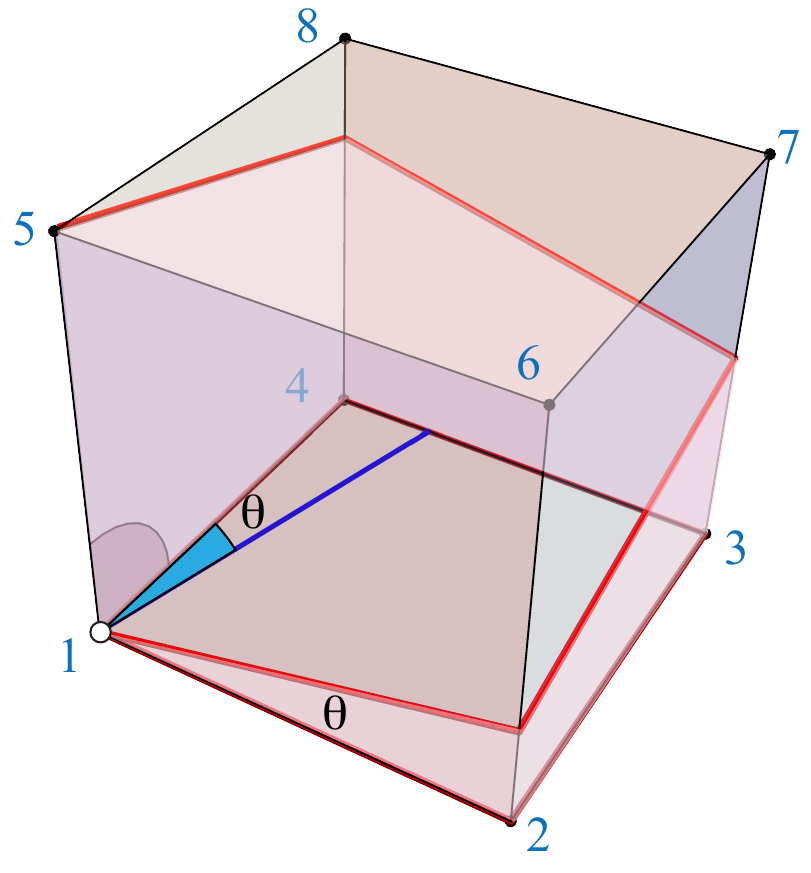}
\caption{The Case~1 F-cone in 3D.
$\q= \arctan(1/4)$.
Cf. Fig.~\protect\figref{BackwardRays}.}
\figlab{Cone3D14}
\end{figure}
%%%%%%%%%%%%%%%%%%%%%%%%%%%%%%%%%Figure End

%\subsection{Case~1}
\paragraph{Case 1.}
We focus on Case~1 in Fig.~\figref{BackwardRays}.
The 2D unfolding of that case is shown on the 3D cube
in Fig.~\figref{Cone3D14}.
Cases~4, 5 and~7 will follow the same general scheme as does Case~1.

% \Anna{This argument should be for the slope $(0, 1/4]$ -- i.e. open at the left, closed at the right.  In particular, the argument works for $1/4$. 
% If we want to argue about slopes (independent of revisiting $F$ (which I recommend)) then this argument does not apply to slope $0$.}

%Assume that there is a geodesic segment $g$ originating at $v_1$ crossing face F,
%and revisiting F a second time. 
View $g$ as directed crossing faces
$\textrm{F}_1$, R, K, L in order.
In Case~1, $g$ has slope in $(0/1,1/4]$ and lies within the pink F-cone of angle $\theta$ as  illustrated.
We now show that $g$ cannot be part
of a quasigeo, by analyzing the 
possibilities for the previous geodesic segment $g'$.

% \anna{
%Consider vertex $v_1$ in face $F_2$.  
Because the angle between $g$ and $g'$ at $v_1$ must be $\le \pi$, $g'$ must leave $v_1$ in a $\theta + \pi/2$ cone that extends counterclockwise from edge $v_1v_5$.  This cone is open along edge $v_1 v_5$ and closed on its other boundary.  See vertex $v_1$ in Face $F_2$ in the figure.  We partition the cone into three possibilities: 
% }

%If $g$ is a component of a quasigeo $Q$, then there must be an ``incoming''
%geodesic segment $g'$ that terminates at $v_1$.
%Now we analyze $g'$ ``backwards,'' starting at $v_1$.

\begin{enumerate}[(1)]
%\item
%If $g'$ starts clockwise of cube edge $v_1 v_5$ (on $\textrm{F}_2$), 
%the turn at $v_1$ between $g'$ and $g$ is too sharp for a quasigeodesic.
\item
$g'$ 
lies strictly within the quarter-circle on face L at $v_1$
(counterclockwise between edges  $v_1 v_5$ and $v_1 v_4$).
%,
Then $g'$ crosses $g$
no matter where $g$ and $g'$ lie in their respective cones.
\item
$g'$ lies in the cone of angle $\theta$ %that lies
% \JOR{Removed repetition of "lies".}
counterclockwise of edge $v_1 v_4$. This cone (colored blue in Fig.~\figref{Cone3D14}) is open along the edge $v_1 v_4$ and closed on its other boundary.
%If $g'$ falls at most $\q$ counterclockwise of edge $v_1 v_4$
%(so at most parallel to the upper F-cone edge),
%it 
Then $g'$ wraps clockwise around $v_4$ by $\pi/2$, and crosses $g$ in face K.

\item
% \JOR{Please check the logic of this step.
% I think we need to consider every possible $g'$, and show that none work.}
% \Anna{I expanded a bit.}
% \JOR{Very clear now---Thanks! I'll duplicate the logic in the other cases.}
%If 
$g'$ follows the edge $v_1 v_4$.
Then $g'$ hits vertex $v_4$ and ends there.  Let $g''$ be the next geodesic segment.  Then $g''$ leaves $v_4$ in 
face K in 
%a closed $\pi/2$ cone 
the closed quarter-circle
bounded by edges $v_4 v_8$ and $v_4 v_3$.  Any $g''$ in this cone intersects $g$ unless $g''$ follows the edge $v_4 v_3$.  Repeating this argument, we either find an intersection with $g$, or we eventually follow the edge $v_2 v_1$---but then the angle with $g$ at $v_1$ is too sharp for a quasigeodesic.
\end{enumerate}
So we obtain a quasigeo violation for every $g$ inside or on the
upper boundary of the F-cone in Case~1.

The argument for the remaining cases proceeds similarly, presented below somewhat more concisely.

%see Fig.~\figref{BackwardRays}.
%%%%%%%%%%%%%%%%%%%%%%%%%%%%%%%%%Figure Begin
\begin{figure}[htbp]
\centering
\includegraphics[width=1.1\textwidth]{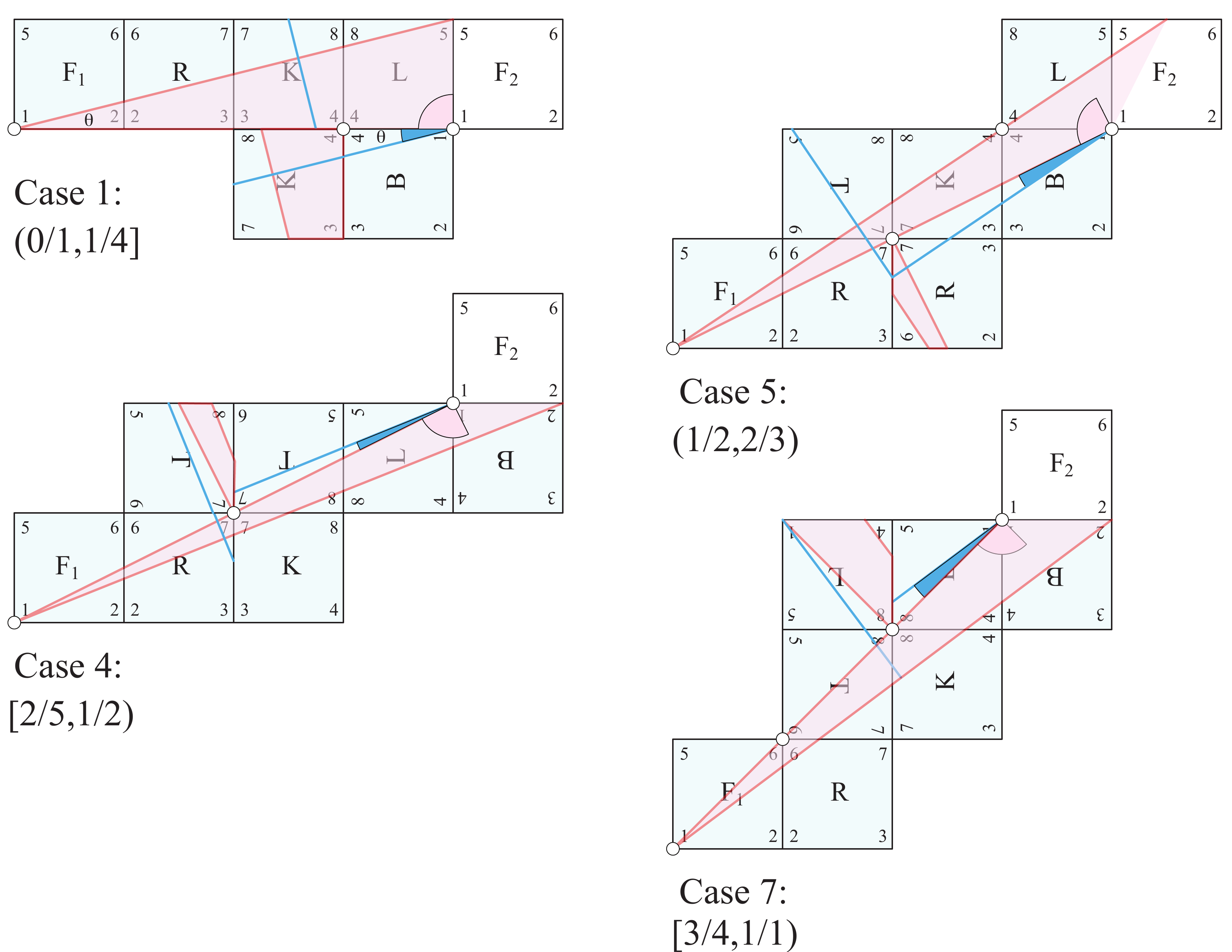}
\caption{The geodesic segment $g$ in the F-cone is crossed by $g'$,
either if starting backwards in the quarter-circle, or starting as much as $\q$ beyond (blue angle and segment.), where $\q$ is the F-cone angle at $v_1$.}
\figlab{BackwardRays}
\end{figure}
%%%%%%%%%%%%%%%%%%%%%%%%%%%%%%%%%Figure End

%\JOR{Each of the next three cases parallel Case~1, but phrased a bit more succinctly.}

%\subsection{Case~4}
\paragraph{Case 4.}
Again the F-cone has angle $\q$ at $v_1$ in $\textrm{F}_1$,
and $g'$ must leave $v_1$ at $\textrm{F}_2$ in a $\q+\pi/2$ cone.
\begin{enumerate}[(1)]
\item $g'$ lies strictly within the quarter-circle
on faces B and L.
Then $g'$ crosses $g$ no matter where they lie in their cones.
\item $g'$ lies in the cone of angle $\q$ strictly clockwise of the upper boundary of the F-cone.
Then $g'$ wraps counterclockwise about $v_7$
and crosses $g$ in face R.
\item $g'$ follows the upper F-cone edge (slope $1/2$).
Then $g'$ hits $v_7$.
As in Case~1, repeating the argument,
the next geodesic segment $g''$
leaves the quarter-circle similarly anchored on $v_7$ and either
crosses $g$ in the F-cone, or hits $v_1$ 
at an angle too sharp for a quasigeodesic.
\end{enumerate}

%\subsection{Case 5}
\paragraph{Case 5.}
\begin{enumerate}[(1)]
\item $g'$ lies strictly within the quarter-circle
on faces B and L.
Then $g'$ crosses $g$ no matter where they lie in their cones.
\item $g'$ lies in the cone of angle $\q$ strictly counterclockwise of the lower boundary of the F-cone.
Then $g'$ wraps clockwise about $v_7$
and crosses $g$ in face R or T.
\item $g'$ follows the lower F-cone edge (slope $1/2$).
Then $g'$ hits $v_7$.
Repeating the arguments of the previous cases,
the next geodesic segment $g''$
leaves the quarter-circle anchored on $v_7$ and either
crosses $g$ in the F-cone, or hits $v_1$ 
at an angle too sharp for a quasigeodesic.
\end{enumerate}

%\subsection{Case 7}
\paragraph{Case 7.}
\begin{enumerate}[(1)]
\item $g'$ lies strictly within the quarter-circle
on faces B and L.
Then $g'$ crosses $g$ no matter where they lie in their cones.
\item $g'$ lies in the cone of angle $\q$ strictly clockwise of the upper boundary of the F-cone.
Then $g'$ wraps counterclockwise about $v_8$
and crosses $g$ in face L or K.
\item $g'$ follows the upper F-cone edge (slope $1/1$).
Then $g'$ hits $v_8$.
We repeat the previous arguments.
The next geodesic segment $g''$
leaves the quarter-circle anchored on $v_8$ and either
crosses $g$ in the F-cone, or hits $v_6$.
Applying the argument again,
the next geodesic segment $g'''$
either crosses $g$ or
or hits $v_1$ 
at an angle too sharp for a quasigeodesic.
\end{enumerate}

This completes the proof of Claim~\ref{claim:backward-cases}.
\end{proof}

\medskip
Claims~\ref{claim:easy-cases} and~\ref{claim:backward-cases}
establish that,
of the seven cases filling the entire range
of slopes (Fig.~\figref{CasesFill45deg}),
all but the five identified slopes are impossible,
and so prove Lemma~\lemref{5Slopes}.
% \Anna{I put this as a Corollary to Lemma 1 above, but maybe we could add there ``this contrasts with . .. ''.} \jor{Note this implies that no geodesic segment
% that is part of a quasigeo on the cube can cross
% a face twice, simply because none of the five can---they are not long enough.
% This contrasts with the long box example,
% Fig.~\figref{LongBoxTurns}.}

\section{Search for Quasigeos}
\seclab{Search}
We initially found the $15$ quasigeos in Fig.~\figref{InventoryPepa15}
``by hand.''  To establish that there are no other possibilities, we programmed an exhaustive search based on Lemma~\lemref{5Slopes}.
%
%Although we found the $15$ quasigeos in Fig.~\figref{InventoryPepa15}
%``by hand,''
%%Lemmas~\lemref{NotTwice} and
%Lemma~\lemref{5Slopes} makes it feasible to check the list by an exhaustive search.
We chose to use a DFS search, starting with the longest geodesic segments
first, because they maximize pruning.
Ordered by lengths, the slopes are $2/3 > 1/3 > 1/2 > 1/1 > 0/1$.
Examples of pruning are shown in Fig.~\figref{DFS_pruning}.

The DFS found $29$ quasigeos, and after eliminating the duplicates congruent by a symmetry,
exactly the $15$ in Fig.~\figref{InventoryPepa15} remain.\footnote{We have not made our code available, but it is an easy programming exercise to verify our exhaustive search.}
%\Anna{We are not providing any way for a person to check our exhaustive search, right?  }
%\JOR{Right. I brought this up at CompGeom once, and Stefan
%said it is OK because the implementation is straightforward
%and any good programmer could write their own code.
%(I believe he said he did this in one of his papers, although I may be misrembering.)
%My code is not in a shareable state
%(I'm really not a good programmer), and I don't want to spend
%the considerable effort to make it shareable.}

% \jor{A consequence of the inventory of the $15$ quasigeos
% is that no quasigeo on a cube crosses a face twice.}
% \Anna{Perhaps label this as an Observation.}

\medskip
Recall
Corollary~\corref{NotTwice} established that no single
geodesic segment part of a cube quasigeo can cross a face more than once.
This contrasts with the long box example, Fig.~\figref{LongBoxTurns}.
A consequence of the inventory of the $15$ quasigeos
is that no cube quasigeo can cross a face more than once.

% \Annachks
% \JOR{Code chks for legal quasi angle, checks for cycle, chks for intersections. Code does not chk for more than one segment on a face.}

%
%see Fig.~\figref{GraphDeg24}.
%%%%%%%%%%%%%%%%%%%%%%%%%%%%%%%%%%Figure Begin
%\begin{figure}[htbp]
%\centering
%\includegraphics[width=1.0\textwidth]{Figures/GraphDeg24}
%\caption{$24$.}
%\figlab{GraphDeg24}
%\end{figure}
%%%%%%%%%%%%%%%%%%%%%%%%%%%%%%%%%%Figure End

%see Fig.~\figref{DFS_pruning}
%%%%%%%%%%%%%%%%%%%%%%%%%%%%%%%%%Figure Begin
\begin{figure}[htbp]
\centering
\includegraphics[width=1.0\textwidth]{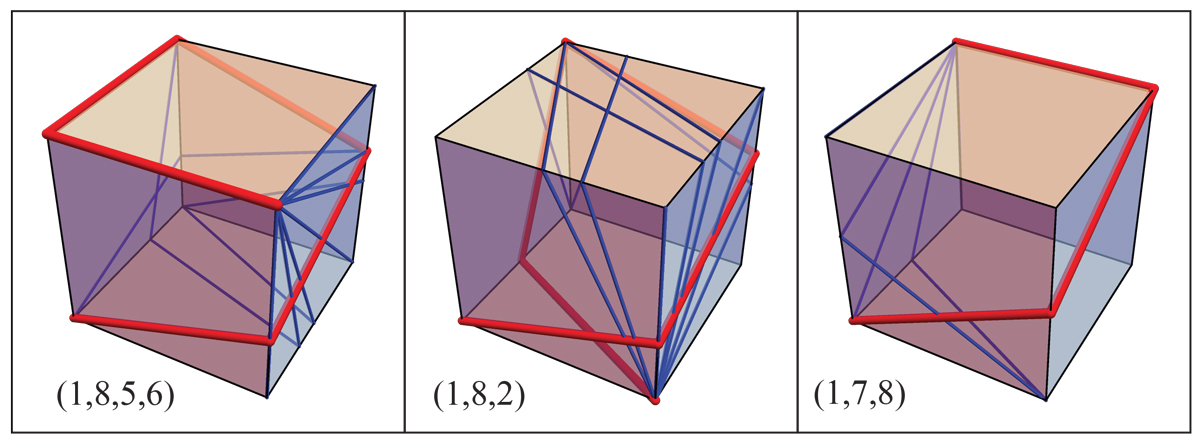}
\caption{Red: Partial quasigeo, through vertices listed. 
Blue: All possible next segments based on angle with the previous segment.} 
\figlab{DFS_pruning}
\end{figure}
%%%%%%%%%%%%%%%%%%%%%%%%%%%%%%%%%Figure End

\section{Discussion and Open Problems}
\seclab{Conclusion}
We have proved Theorem~\thmref{main15}
by verifying that the list in Fig.~\figref{InventoryPepa15} is exhaustive.
%
%\pepa{Say something like ``Below we list several open questions.''?}
Below we list several open questions.
\begin{itemize}
\item
Is there a finite upper bound to the number of simple closed quasigeodesics (that are not geodesics) 
on a given nondegenerate polyhedron of $n$ vertices?
There is no such bound for simple closed geodesics.
Nor is there a bound for (degenerate) doubly-covered squares:
see Fig.~\figref{SquareSpiral}.
%%%%%%%%%%%%%%%%%%%%%%%%%%%%%%%%%Figure Begin
\begin{figure}[htbp]
\centering
\includegraphics[width=0.75\textwidth]{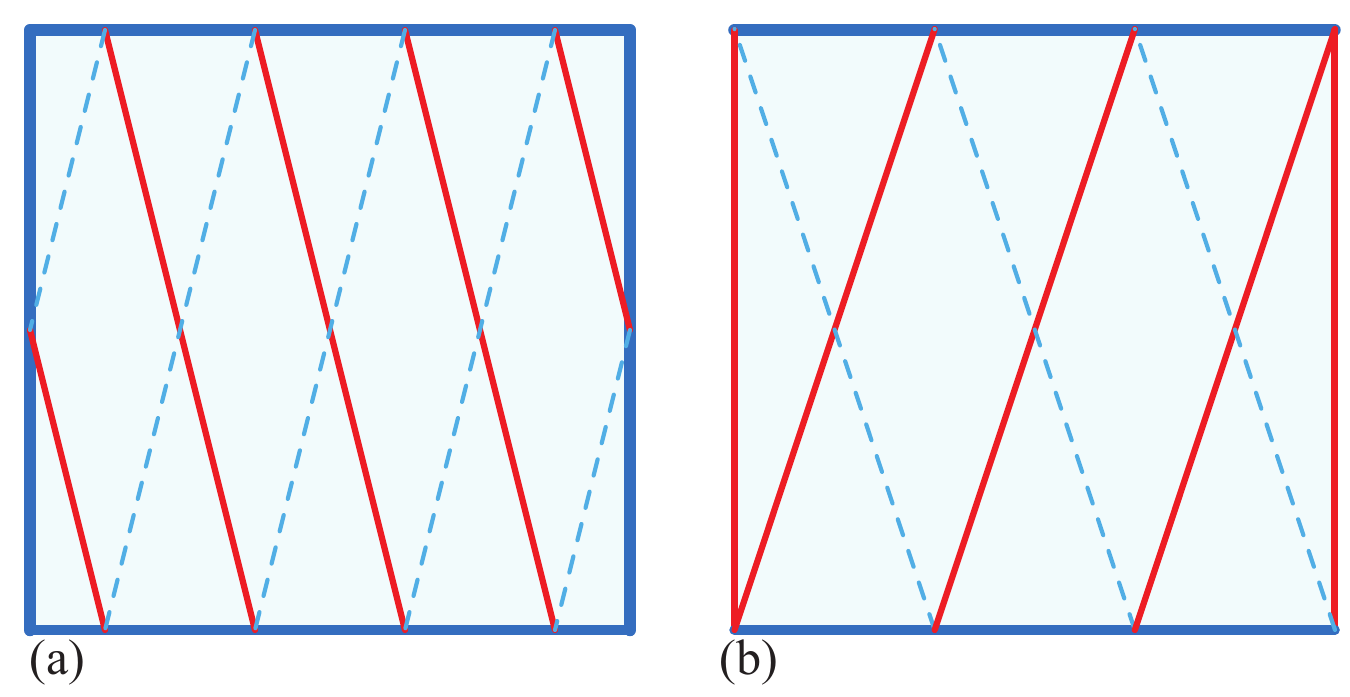}
\caption{Doubly-covered square. 
Red segments: front. Blue segments: back.
(a)~Simple closed geodesic. (b)~Simple closed quasigeodesic.}
\figlab{SquareSpiral}
\end{figure}
%%%%%%%%%%%%%%%%%%%%%%%%%%%%%%%%%Figure End

%
\item It was proved in~\cite{QonT} that every tetrahedron has 
a simple closed geodesic or a $1$-vertex simple closed quasigeodesic.
That the same holds for any convex polyhedron was conjectured in~\cite{Reshaping}. 
As mentioned, it is known
from~\cite{davis2017geodesics} that the cube does not have a $1$-vertex simple
closed quasigeodesic, but it does have simple closed geodesics,
so the cube accords with the conjecture.
Settling the conjecture either way seems currently out of reach.
\item A slightly non-cubical box,
$1 \times 1 \times h$ for $h \in (1,2)$,
has a ``diamond'' $1$-vertex simple closed quasigeodesic:
see Fig.~\figref{PepaDiamond}.
%%%%%%%%%%%%%%%%%%%%%%%%%%%%%%%%%Figure Begin
\begin{figure}[htbp]
\centering
\includegraphics[width=0.75\textwidth]{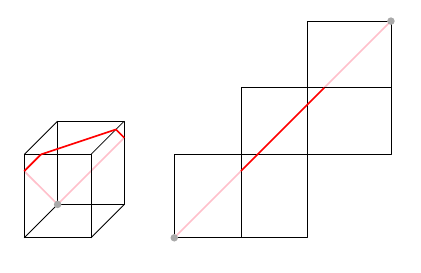}
\caption{A $1$-vertex quasigeo on a $1 \times 1 \times 1\frac{1}{4}$ box.}
\figlab{PepaDiamond}
\end{figure}
%%%%%%%%%%%%%%%%%%%%%%%%%%%%%%%%%Figure End
Characterizing all simple closed quasigeodesics on boxes is a natural next step.
%seems like it might be challenging.
%
\end{itemize}

\clearpage
\bibliographystyle{alpha}
\bibliography{/Users/jorourke/Documents/geom}

\end{document}